\documentclass[11pt]{article}

\usepackage{amsmath,amssymb,amsfonts, amsthm, amsbsy, mathtools, epsfig}
\usepackage[usenames]{color}
\usepackage{rotating}
\usepackage{enumerate}

\usepackage{verbatim}
\usepackage{caption}
\usepackage{subcaption}
\usepackage{bbm}
\usepackage{xcolor}
\usepackage{bm}
\usepackage{mathrsfs,color,dsfont}

\usepackage{thmtools}
\usepackage{thm-restate}
\usepackage{hyperref}
\usepackage{cleveref}

\newtheorem{thm}{Theorem}
\newtheorem{corollary}[thm]{Corollary}

\newtheorem{theorem}{Theorem}
\newtheorem{lemma}[theorem]{Lemma}
\newtheorem{proposition}[theorem]{Proposition}

\newcommand{\Pb}{\text{Pb}}
\newcommand{\Sc}{\mathcal{S}}
\newcommand{\Ec}{\mathcal{E}}

\newcommand{\M}{{\mathbb{M}}}

\newcommand{\RNum}[1]{\uppercase\expandafter{\romannumeral #1\relax}}
\makeatother

\newcommand\tabcaption{\def\@captype{table}\caption}

\definecolor{DSgray}{cmyk}{0,1,0,0}

\newcommand{\E}{\mathbb{E}}

\newcommand{\PP}{\mathbf{P}} 

\newcommand{\Th}{\text{Th}}
\newcommand{\Ab}{\text{Ab}}
\newcommand{\Serv}{\text{Serv}}

\begin{document}





\title{On the SRPT Scheduling Discipline in Many-Server Queues with Impatient Customers}

\author{Jing Dong\footnote{Columbia University, Email: jing.dong@gsb.columbia.edu} \footnote{ Support from NSF grant CMMI-1944209 is gratefully acknowledged by J. Dong} ~ and Rouba Ibrahim\footnote{University College London, Email: rouba.ibrahim@ucl.ac.uk}}
\date{}

\maketitle

\begin{abstract}
The shortest-remaining-processing-time (SRPT) scheduling policy has been extensively studied, for more than 50 years, in single-server queues with infinitely patient jobs. Yet, much less is known about its performance in multiserver queues. In this paper, we present the first theoretical analysis of SRPT in multiserver queues with abandonment. In particular, we consider the $M/GI/s+GI$ queue and demonstrate that, in the many-sever overloaded regime, performance in the SRPT queue is equivalent, asymptotically in steady state, to a preemptive two-class priority queue where customers with short service times (below a threshold) are served without wait, and customers with long service times (above a threshold) eventually abandon without service. We prove that the SRPT discipline maximizes, asymptotically, the system throughput, among all scheduling disciplines. We also compare the performance of the SRPT policy to blind policies and study the effects of the patience-time and service-time distributions.
\end{abstract}




%


\section{Introduction}

In this paper, we study scheduling decisions in multiserver queues with processing-time information. In particular, we assume that the service requirements of individual jobs (customers) are known upon arrival and may be used to determine the order in which these jobs are processed. When processing times are known, it seems natural to give priority to the jobs with the shorter remaining processing times in order to minimize the mean sojourn time, i.e., the time from arrival until departure from the system. Indeed, a large body of literature shows that the \textbf{s}hortest-\textbf{r}emaining-\textbf{p}rocessing-\textbf{t}ime (SRPT) policy has, in general, superior performance; see, for example, \cite{schrage1966queue}, \cite{lin2011heavy}, and \cite{HTSRPT2020}, to name a few.

The SRPT policy has been extensively studied for over 50 years, yet exclusively in single-server queues with infinitely patient jobs. There is only one exception: The recent paper \cite{grosof2018srpt} studies the performance of the SRPT policy in a multiserver queueing system with Poisson arrivals, general service times, and no abandonment. Grosof et. al. \cite{grosof2018srpt} prove that the SRPT policy achieves an asymptotically optimal mean sojourn time in the conventional heavy traffic regime, i.e., where the number of servers in a sequence of multiserver queues is held fixed, while the arrival rates along that sequence approach the total service capacity. 

In settings where jobs are human customers to be scheduled, e.g., in service systems, it is important to account for finite customer patience times. That is, customers do not wait indefinitely for service, and they abandon the queue if their waiting time exceeds their patience time. For example, there is empirical evidence substantiating finite customer patience in emergency departments \cite{batt2015waiting} and call centers \cite{brown2005statistical}. Moreover, it is well known that incorporating customer impatience strongly affects performance in the system \cite{garnett2002designing}. Thus, there is a need to investigate whether the superior performance of SRPT continues to hold in multiserver queues where patience times are finite. This investigation is the focus of our paper.

\noindent \textbf{What is this paper about?} To the best of our knowledge, the performance of SRPT in multiserver queues with abandonment is entirely open. In this paper, we take a step towards filling that gap in the literature, by studying {\it the steady-state performance of SRPT in the $M/GI/s+GI$ queue.} We adopt a many-server asymptotic mode of analysis and focus on the overloaded regime, which is also known as the efficiency-driven regime. This regime is appropriate because queueing times are negligible, in large systems with abandonment, under moderate or light load (i.e., critically loaded or underloaded regimes) \cite{garnett2002designing}. In the many-server overloaded regime, a non-negligible proportion of customers abandon the queue. Thus, carefully designing the scheduling policy to optimize the throughput is crucial in this setting.  

For multiserver queues with abandonment, under SRPT scheduling, we demonstrate a state-space collapse in the many-server overloaded limit. In particular, we prove that only customers with long service times (above a threshold) wait in the queue, and eventually abandon, whereas customers with short service times are immediately served.
We also derive closed-form expressions for key steady-state performance measures in the limit. We prove that, asymptotically, among all scheduling policies, SRPT maximizes the throughput in the system, minimizes the expected waiting time conditional on being served, and maximizes the expected waiting time conditional on abandoning. 
We focus on such measures, rather than the mean sojourn time as is common in the extant literature, because not all customers are served in queues with abandonment. We also show that performance in the SRPT queue is, asymptotically, insensitive to the patience-time distribution beyond its mean, which is unlike  performance under first-come-first-served (FCFS). 
%

We compare SRPT to blind policies that do not use the processing-time information, such as FCFS and last-come-first-served (LCFS). In addition to throughput, we also compare the expected waiting times. Based on fluid approximations, we show that, when the patience-time distribution has a non-decreasing failure rate, SRPT yields a smaller expected waiting time than any blind policy. On the other hand, when the patience-time distribution has a decreasing failure rate, SRPT yields a smaller expected waiting time than LCFS, but it may lead to longer expected waiting time than FCFS. Either way, even when SRPT beats blind policies, it does not offer any order of magnitude improvement in waiting times over those policies. This lies in contrast to the asymptotic system performance in the conventional heavy-traffic regime, as the traffic intensity increases \cite{lin2011heavy, puha2015diffusion, chen2020power}.

\noindent \textbf{Why is this problem difficult?} In general, analyzing SRPT is complicated because it requires keeping track of the remaining processing time of each customer in the system. Even in single-server queues, where closed-form expressions have been known for a while, comparing SRPT to other scheduling disciplines is difficult because those closed-form expressions, e.g., for the mean sojourn time, have complex forms and involve nested integrals. Asymptotic analysis, e.g., under heavy traffic, generally allows for simpler descriptions of the system. However, the asymptotic analysis of SRPT involves studying suitably scaled measure-valued system state descriptors, which imposes substantial technical challenges \cite{HTSRPT2020}.

Even without abandonment, when moving from a single server to multiple servers, there is a main challenge in extending the existing single-server arguments. The difficulty arises from the fact that multiserver queues are not work-conserving. Specifically, this makes the analysis of busy periods and steady-state workload, both of which are central to the ``tagged job approach'' of analysing SRPT single-server queues \cite{schrage1966queue}, difficult to extend to a multiserver setting; see section 4.2 in \cite{grosof2018srpt}.

In this paper, we allow for multiple servers, general service times, and general patience times, which complicates the analysis even more. Scheduling decisions in systems with abandonment is notoriously difficult, because the optimal scheduling policy can be complex and depends on the patience-time distribution \cite{puha2019scheduling}. For example, when the system is critically loaded, the optimal diffusion control may no longer follow a simple fixed priority rule \cite{kim2013dynamic, kim2018dynamic}. In this paper, we circumvent the difficulty of doing direct analysis on the SRPT queue by relying on a {properly coupled loss queueing system}; see section \ref{asymp} for details.

\noindent \textbf{Literature review. } 
 Because of its optimality properties, the study of SRPT in single-server queues has been the topic of hundreds of papers. Given the richness of that literature, we do not attempt to be comprehensive in our review, and only mention a few key references instead.

The works \cite{schrage1966queue} and \cite{schrage1968letter}  demonstrate optimality properties of SRPT in the $M/G/1$ system. There is a notable stream of works that studies SRPT under heavy-traffic \cite{down2009fluid, gromoll2011diffusion, puha2015diffusion}. 
Scully et.\ al.\ \cite{scully2018soap} develop a unified framework to analyze several age-based scheduling policies in the $M/G/1$ queue.

Recently, Chen and Dong \cite{chen2020power} demonstrate that in the $GI/GI/1$ queue, under heavy traffic, a preemptive two-class priority rule achieves asymptotically comparable performance to the SRPT policy. In the two-class priority rule, customers whose service times are shorter than a certain threshold are given preemptive priority over customers whose service times are above that threshold. 
They establish state-space collapse, under which only the low-priority customers (with long service times) occupy the queue in the heavy-traffic limit. 
Similar results are common in scheduling multi-class priority queues; e.g., see \cite{reiman1984some}, \cite{bramson1998state}, and \cite{dai2011state}.


In stark contrast to the single-server setting, much less is known about the performance of SRPT in a multiserver queueing model. With multiple servers, SRPT is not necessarily optimal \cite{leonardi2007approximating}. However, Grosof et. al. \cite{grosof2018srpt} recently show that it is asymptotically optimal under heavy load. In \cite{grosof2018srpt}, jobs are assumed to be infinitely patient. In this paper, we consider finite patience times. 
To the best of our knowledge, we are the first to derive theoretical results on the performance of SRPT in multiserver queues with abandonment.


\noindent \textbf{Paper organization.} 
The rest of this paper is organized as follows. In section \ref{model}, we describe our modeling framework. In section \ref{asymp}, we derive our main result on the asymptotic equivalence of the SRPT queue with a preemptive two-class priority system. In section \ref{perf}, we compare the performance of SRPT to blind policies and study the roles of the service-time and patience-time distributions. In section \ref{conc}, we draw conclusions. We relegate some technical proofs to the appendix. 


\section{Modeling Setup: The SRPT $M/GI/s+GI$ Queue} \label{model}

In this section, we set the stage for our subsequent theoretical development by describing our modeling framework and defining our many-server asymptotic mode of analysis.

\subsection{Model Description}

We consider the $M/GI/s+GI$ queue in steady state, i.e., we assume that the arrival process is Poisson with rate $\lambda$, service times are independent and identically distributed (i.i.d.) with a general cumulative distribution function (cdf) $G$ and mean 1/$\mu$, and times to abandon are i.i.d. with cdf $F$ and mean $1/\theta$. Let $S$ denote a generic service time and $T$ a generic patience time. In addition, let $S_i$ and $T_i$ denote the service time and patience time of the $i$-th arriving customer. We assume that the service-time distribution and the patience-time distribution are continuous with probability density functions $g$ and $f$, respectively. There are $s$ homogeneous servers working in parallel. 


We consider the SRPT queueing discipline. Specifically, a customer who arrives to find an empty server goes to service immediately upon arrival. If all servers are busy at the arrival epoch, and there exists at least one customer in service whose remaining processing time is longer than the new arriving customer's, then the customer in service with the longest remaining processing time is preempted and joins the queue. Otherwise, the new arrival joins the queue directly.
Customers have finite patience times, generated at the arrival epoch of the customer. If the cumulative amount of the time that the customer spends in the queue exceeds her patience time, then the customer abandons the system. In particular, if a customer enters service and is later preempted back to queue, then we assume that her initial patience time (which had not fully elapsed since she did not abandon previously) continues to elapse, i.e., we do not generate a new patience time for the preempted customer at every preemption epoch. We assume that the arrival, service, and abandonment processes are mutually independent. We define the traffic intensity $\rho \equiv \lambda/s\mu$.

Because abandonment is allowed in the system, it is not necessary to assume $\rho < 1$ for the system to reach a steady state. To elaborate, with general service-time or patience-time distributions, there is no finite-dimensional Markovian representation of the queue \cite{dai2013many}. Indeed, a Markovian description of the state of the system would require keeping track of the remaining or elapsed patience times and the remaining or elapsed service times of each customer present in the system (in service or in queue). We now present an infinite-dimensional state representation which leads to a Markovian description of the dynamics of the system (with respect to a suitable filtration). For $t\geq 0$, let $X(t)\in \mathbb{N}_0:=\{0,1,\dots\}$ denote the number of customers in the system
and $R(t)\in \mathbb{R}^{3\times\infty}$ denote the remaining service times, remaining patience times, and initial service times of customers in the system. Let $R_{j}(t)$ be the $j$-th column of $R(t)$. When $X(t)>0$, for $j\in \{1, \dots, X(t)\}$, $R_j(t)$ is a column vector whose first element is the remaining service time, the second element is the remaining patience time, and the third element is the initial service time of the $j$-th earliest arriving customer, among all customers currently in the system. For $j>X(t)$, $R_j(t)$ is a column vector with zero entries. Then, the process $R(t)$ is a Markov process which describes system dynamics. Note that, in order to describe how the system evolves, we only need to know the first two elements in $R_j(t)$, i.e., the remaining service time and the remaining patience time of each customer. We add a third element, the initial service time, to facilitate the development of the state-space collapse result. The invariant measure of the Markov process $R(t)$ is referred to as the steady-state distribution of the system. 

\subsection{Many-Server Overloaded Regime} \label{scaling}
To derive theoretical insights on the performance of SRPT, we consider a sequence of $M/GI/s_\lambda+GI$ queues, indexed by the arrival rate $\lambda$. We fix the traffic intensity in system $\lambda$ to $\rho_\lambda = \lambda/(s_\lambda\mu) \equiv \rho > 1$, i.e., we consider an overloaded setting. We hold the service-time and patience-time distributions fixed, independently of $\lambda$, and let $\lambda$ and $s_\lambda$ increase without bound. 

Define the threshold, $\tau$, satisfying 
\begin{equation} \label{eq:t} \lambda \cdot \mathbb{P}(S \leq \tau) \cdot \mathbb{E}[S|S\leq \tau] = \lambda \mathbb{E}[S \mathbf{1}(S \leq \tau)] = s_\lambda, \end{equation} 
where $\mathbf{1}(\cdot)$ denotes the indicator function. That is, we choose $\tau$ such that
the total workload of customers with service times smaller than or equal to $\tau$ matches the service capacity of the system. 
In the following section, we prove state-space collapse, i.e., that the SRPT queue is asymptotically equivalent to a two-class priority queue.
The high-priority class, defined as jobs whose service times are smaller than or equal to $\tau$, has preemptive priority over the low-priority class, defined as jobs whose service times are larger than $\tau$. We emphasize that $\tau$, as defined in \eqref{eq:t}, does not depend on $\lambda$ since $s_\lambda/\lambda = 1/(\rho \mu)$ is held fixed under our scaling.

\section{Asymptotic Analysis} \label{asymp}

Direct analysis on the SRPT $M/GI/s_\lambda+GI$ queue is complicated, which partly explains why it has eluded theoretical analysis for decades. Here, we overcome the technical challenges by proposing a coupling argument. Specifically, we derive asymptotic results quantifying performance in the SRPT queue by constructing a coupled two-class preemptive priority $M/GI/s_\lambda/s_\lambda$ (loss) queue, with the same arrival process and service-time distribution. Under the coupling, both systems see the same arriving customers, i.e., with the same arrival times and service requirements. However, the service disciplines in the two systems are different. In contrast to multiserver queues with abandonment under SRPT, much more is known about loss queues. 

\subsection{A Sequence of Coupled Loss Systems}

The coupled $M/GI/s_\lambda/s_\lambda$ loss queue operates under the following preemptive two-class priority rule. Recall the threshold $\tau$ defined in \eqref{eq:t}. In the loss queue, customers whose service times are less than or equal to $\tau$ are grouped into the high-priority class, which we refer to as class 1.
The remaining customers are grouped into the low-priority class, which we refer to as class 2. There are $s_\lambda$ servers, and no waiting room. 

Upon arrival, a high-priority, class 1, customer enters service immediately if there is an empty server, or if there is at least one low-priority, class 2, customer in service (the class 2 customer with the longest remaining processing time will be preempted). Otherwise, i.e., if all servers are busy with class 1 customers, the newly arriving customer is {lost}. A class 2 customer enters service only if there is an idle server upon arrival and, if all servers are busy, the class 2 customer is lost. A preempted class 2 customer is also lost as there is no waiting room in the system. Because of preemption, the high-priority customers do not ``see'' the low-priority customers: For class 1 customers, the system behaves as a single-class $M/GI/s_\lambda/s_\lambda$ queue with arrival rate $\lambda G(\tau)$, resulting from thinned Poisson arrivals, and service time distribution $[S|S\leq \tau]$.

In what follows, we refer to the $M/GI/s_\lambda+GI$ SRPT queue as system $O$, where $O$ stands for {\bf o}riginal, and the two-class priority $M/GI/s_\lambda/s_\lambda$ loss queue as system $L$, where $L$ stands for {\bf l}oss. By a slight abuse of notation, in both systems, we refer to customers whose service times are less than or equal to $\tau$ as class 1 customers, and customers whose service times are longer than $\tau$ as class 2 customers.

For $\M\in\{L,O\}$, $i=1,2$, and $t\geq 0$, we let $N_{\M,i}(t)$ denote the number of class $i$ customers served (have successfully finished service) in system $\M$ by time $t$. 
We also define
\begin{equation} \label{eq:th}
\Th_{\M,i}=\lim_{t\rightarrow\infty}\frac{N_{\M,i}(t)}{t}.
\end{equation}
$\Th_{\M,i}$ is the long-run departure rate (throughput) of class $i$ customers in system $\M$. Let $\Th_{O}=\Th_{O,1}+\Th_{O,2}$ to be the total throughput in system $O$.

The following proposition builds the foundation of our coupling argument and is the key to deriving limiting performance measures in section \ref{Asymperf}.

\begin{proposition} \label{prop:main}
For the coupled SRPT and loss queues, and the throughput defined as in \eqref{eq:th}:
\begin{equation} \label{boundTh} \Th_{L,1} \leq \Th_{O}.\end{equation}
\end{proposition}

\begin{proof}
We consider two coupled systems which are empty initially and see exactly the same customers. We next introduce a mechanism to match each class 1 customer served in system $L$ to a customer in system $O$ who finishes service no later than the customer in system $L$. In particular, we will show that, at any time $t$, for each class 1 customer who is in service in system $L$, there is a matched customer who is either in service with an equal or shorter remaining processing time in system $O$, or has already finished service in system $O$. In addition, each customer in system $O$ is matched with at most one class 1 customer in system $L$. We prove this claim and construct the matching by induction on the consecutive arrival epochs of class 1 customers.

The above claim is trivially true before the arrival of the first class 1 customer. For the inductive step, we suppose that it is true before the arrival of the $k$-th class 1 customer, and proceed to show that it holds after that arrival. We refer to the $k$-th class 1 customer as customer $k$. If customer $k$ is lost in system $L$, then the claim is trivially true. If customer $k$ enters service in system $L$ (by either joining an empty server or preempting a class 2 customer in service), then we consider the following three scenarios that can happen in system $O$ upon the arrival of customer $k$.

{\bf Case I.} In system $O$, customer $k$ joins an empty server or preempts a customer who is not matched with any customer in system $L$. Then, we can match customer $k$ in system $L$ with customer $k$ in system $O$. 

{\bf Case II.} In system $O$, customer $k$ preempts a customer who is already matched with a class 1 customer in system $L$. We refer to this preempted customer in system $O$ and the matched customer in system $L$ as customer $k_O^{\prime}$ and customer $k_L^{\prime}$ respectively. Note that in this case, before the arrival of customer $k$, system $L$ has strictly less than $s^{\lambda}$ class 1 customers in service, while there are $s^{\lambda}$ customers in service in system $O$. Based on the inductive assumption, in system $O$, there must be a customer in service that has not been matched with any class 1 customer in system $L$ yet. We refer to this customer as customer $k_O^{\prime\prime}$. As customer $k_O^{\prime}$ is preempted instead of customer $k_O^{\prime\prime}$, customer $k_O^{\prime\prime}$ must have a shorter remaining processing time than customer $k_O^{\prime}$. First, we rematch customer $k_L^{\prime}$ in system $L$ with customer $k_O^{\prime\prime}$ in system $O$. Note that as the remaining processing time of customer $k_L^{\prime}$ is larger than customer $k_O^{\prime}$ based on our inductive assumption, customer $k_O^{\prime\prime}$ will finish service before customer $k_L^{\prime}$. Then, we match customer $k$ in system $L$ with customer $k$ in system $O$. 

{\bf Case III.} In system $O$, customer $k$ waits in the queue. Similar to Case II, in this case, before the arrival of customer $k$, system $L$ has strictly less than $s^{\lambda}$ class 1 customers in service, while there are $s^{\lambda}$ customers in service in system $O$. Based on the inductive assumption, in system $O$, there must be a customer who has not been matched with any class 1 customer in system $L$ yet. We refer to this customer as customer $k^{\prime}_O$. In addition, this customer must have a shorter remaining processing time than customer $k$. In this case, we match customer $k$ in system $L$ with customer $k_O^{\prime}$ in system $O$.

Under the matching mechanism described above, each class 1 customer who gets served in system $L$ is matched with a customer who gets served in system $O$, and this matched customer in system $O$ is not matched with any other customers in system $L$. Thus, $N_{L,1}(t)\leq N_{O,1}(t)+N_{O,2}(t)$ for each $t \geq 0$, sample path by sample path. Then, we have
\[\lim_{t\rightarrow\infty} \frac{N_{L,1}(t)}{t} \leq \lim_{t\rightarrow\infty} \frac{N_{O,1}(t)+N_{O,2}(t)}{t},\]
i.e., $\Th_{L,1} \leq \Th_{O}$, as desired.
\end{proof}

\subsection{Asymptotic Performance in the SRPT $M/GI/s+GI$ Queue} \label{Asymperf}

In what follows, we consider a particular ``tagged'' customer $c$ arriving to a random system state drawn from the system's steady-state distribution. We denote by ``$\Serv_c$'' the event that customer $c$ is served and ``$\Ab_c$'' the event that customer $c$ abandons the system. We also write $V_c$ as the virtual waiting time assuming customer $c$ has infinite patience, and $W_c$ as her actual waiting time, i.e., $W_c=\min\{V_c, T_c\}$, where we recall that $T_c$ is the patience time of customer $c$. 

In Theorem \ref{thm:srpt}, we derive limits for several key performance measures of the $M/GI/s^{\lambda}+GI$ SRPT queue in steady state. The main observation is the state-space collapse. In particular, in the many-server limit, all class 1 customers are served immediately upon arrival and no class 2 customers are served, i.e., class 2 customers all abandon the queue.

\begin{theorem} \label{thm:srpt}
For the sequence of $M/GI/s^\lambda+GI$ queues under SRPT with $\rho^\lambda = \lambda/s_\lambda\mu > 1$ held fixed and the threshold, $\tau$, as defined in \eqref{eq:t}:
\begin{enumerate} 
\item[(a)] $\lim_{\lambda \rightarrow \infty} \PP(\Serv_c^{\lambda}| S_c \leq \tau) = 1 \mbox{ and }
\lim_{\lambda \rightarrow \infty} \PP(\Serv_c^{\lambda}| S_c > \tau) = 0$.
\item[(b)] $\lim_{\lambda \rightarrow \infty} \E[W_c^{\lambda}|\Serv_c] = \lim_{\lambda \rightarrow \infty} \E[W_c^{\lambda}|S_c\leq\tau]= 0$.
\item[(c)] $\lim_{\lambda \rightarrow \infty}\E[W_c^{\lambda}|\Ab_c] = \lim_{\lambda \rightarrow \infty}\E[W_c^{\lambda}|S_c > \tau] = 1/\theta.$
\item[(d)] $\lim_{\lambda \rightarrow \infty} \E[W_c^{\lambda}] = (1-G(\tau))/\theta.$
\end{enumerate}
\end{theorem}

The proof of Theorem \ref{thm:srpt} can be found in Appendix \ref{pf:thm1}. We note from the theorem that the steady-state probability of abandonment and various expected waiting-time measures are insensitive to the patience-time distribution beyond its mean, and depend solely on the service-time distribution. This lies in contrast to performance in the $M/GI/s+GI$ queue under FCFS, where the system's performance depends on the patience-time distribution beyond its mean \cite{whitt2005engineering, whitt2006fluid}.

The following corollary follows directly from Theorem \ref{thm:srpt} parts (b) and (c). It demonstrates the desirable performance of SRPT.

\begin{corollary}\label{cor:SRPT}
For the sequence of $M/GI/s^\lambda+GI$ queues, SRPT asymptotically minimizes the steady-state waiting time conditional on being served, and asymptotically maximizes the steady-state waiting time conditional on abandoning
\end{corollary}

Let $\Th_M^{\lambda}$ denote the maximum throughput of the $M/GI/s^{\lambda}+GI$ queue. While the maximum throughput is not necessarily achieved by SRPT, we show in the following proposition that SRPT maximizes the throughput asymptotically.
\begin{proposition} \label{lm:th}
For the sequence of $M/GI/s^\lambda+GI$ queues,
\[\lim_{\lambda\rightarrow\infty} \frac{\Th_O^{\lambda}}{\lambda}=\limsup_{\lambda\rightarrow\infty} \frac{\Th_M^{\lambda}}{\lambda}=G(\tau),\]
i.e., SRPT asymptotically maximizes the throughput among all service disciplines.
\end{proposition} 

The proof of Proposition \ref{lm:th} can be found in Appendix \ref{app:prop2}. We note from the proposition that for a fixed value of $\rho$, the asymptotically maximal throughput depends on the service-time distribution. This is in contrast to the throughput of the $M/GI/s^{\lambda}+GI$ queue under FCFS (or LCFS) where, for a fixed value of $\rho>1$, the throughput scaled by $\lambda$ is equal to $1/\rho$ \cite{whitt2006fluid}. Furthermore, from the definition of $\tau$ in \eqref{eq:t}, we have
\[\lambda G(\tau)\E[S|S\leq\tau]=s_{\lambda}=\frac{\lambda \E[S]}{\rho},\]
which implies that
\[G(\tau)=\frac{1}{\rho}\frac{\E[S]}{\E[S|S\leq\tau]}.\]
Because $\rho>1$, $G(\tau)<1$ (since not all customers can be served) and $\E[S|S\leq \tau]<\E[S]$ (since $G(\tau) < 1$). Then, we must have the inequality: 
\begin{equation}\label{eq:inequal}
G(\tau)>1/\rho,
\end{equation}
i.e., the throughput of SRPT is strictly larger than the throughput of FCFS in the many-server overloaded regime.

We will discuss the effect of the service-time distribution on the throughput of SRPT in more detail in section \ref{STDist}: We will illustrate through numerical examples that for fixed values of $\E[S]$ and $\rho$, the heavier the tail of the service-time distribution, the larger the throughput that SRPT can achieve.

\section{Comparison to Blind Policies} \label{perf}

In this section, we compare the performance of SRPT to blind policies that do not use the service-time information, such as FCFS and LCFS. We focus on the effect of patience-time and service-time distributions on system performance. In addition to throughput, we also consider the steady-state expected waiting time. We do so because (1) waiting time measures are generally of interest in the management of service systems, and (2) while SRPT asymptotically maximizes the throughput in the system (Proposition \ref{lm:th}), it does not necessarily minimize waiting times, and it is important to shed further light on this point. 

To compare the performance of SRPT to blind policies, we rely on steady-state fluid approximations for systems under blind policies as described in \cite{whitt2006fluid}. 
Fluid approximations are known to be remarkably accurate in large-scale overloaded systems \cite{kang2010fluid}, which is the regime that we consider (section \ref{scaling}). Throughout this section, we fix the traffic intensity $\rho=1.4$.

\subsection{The Effect of the Patience-Time Distribution} \label{fluidSection}

We consider the asymptotic throughput, scaled by $\lambda$, under the many-server overloaded scaling. 
In this case, the patience-time distribution has no effect on the throughput under any of the 
scheduling disciplines. For the SRPT queue, the throughput is equal to $G(\tau)$ by Proposition \ref{lm:th}. For all blind policies (including FCFS and LCFS), it is equal to $1/\rho$ \cite{whitt2006fluid}.

As for the steady-state waiting time, even though the patience-time distribution beyond its mean has no effect on the limiting steady-state waiting time in the SRPT queue (part (d) in Theorem \ref{thm:srpt}), it plays a central role in the performance of the FCFS queue \cite{whitt2006fluid}. In particular, in the many-server overloaded limit, the steady-state fluid waiting time of the FCFS queue is
\[
\E[T\mathbf{1}(T\leq \bar w)] + \bar w(1-F(\bar w)) \mbox{ where } \bar w= F^{-1}(1-1/\rho).
\]

We first consider patience-time distributions  with strictly increasing-hazard-rate  (IHR)  or strictly decreasing-hazard-rate (DHR). 
The hazard rate of the patience-time distribution is defined as
$h(x)=f(x)/(1-F(x))$. The distribution has IHR (DHR) if $h$ is monotonically increasing (decreasing) in $x$ on $(0,\infty)$.
For blind policies, if $h$ is decreasing in $x$ (i.e., DHR), waiting customers become increasingly patient with time. To minimize the waiting time, we should process customers who waited more first. Analogously, if $h(x)$ is increasing in $x$ (i.e., IHR), then we should process customers who waited less first.

For IHR or exponential patience-time distributions (which has a constant hazard rate), LCFS minimizes the steady-state fluid waiting among all blind policies (Proposition 3 in \cite{bassamboo2016scheduling}).
The fluid LCFS queue is described by two classes: The high-priority class is entirely served and does not wait for service, whereas the low priority class abandons entirely. Thus, the steady-state fluid waiting time under LCFS is given by $(1-1/\rho)\theta$.
On the other hand, Theorem \ref{thm:srpt} shows that the limiting steady-state expected waiting time under SRPT is $(1-G(\tau))/\theta$.
Because $G(\tau)>1/\rho$ for $\rho>1$ (see \eqref{eq:inequal}), $(1-G(\tau))/\theta<(1-1/\rho)\theta$.
This implies that SRPT outperforms all blind policies for overloaded systems with IHR or exponential patience-time distributions, when the system is large enough. For DHR patience-time distributions, FCFS minimizes the steady-state fluid waiting time among all blind policies (Proposition 3 in \cite{bassamboo2016scheduling}).
In this case, SRPT may lead to a larger expected waiting time than FCFS.

In Figure \ref{SRPTWorse}, we compare the performance of FCFS, LCFS, and SRPT under Weibull (right-side figure) or Pareto (left-side figure) patience-time distributions. For both distributions, we vary the shape parameter and adjust the scale parameter accordingly so that the mean time to abandon is equal to 1. We fix the service-time distribution to be exponential with mean equal to 1 as well. 
For LCFS and FCFS, we present the fluid limit.
For SRPT, we present $(1-G(\tau))/\theta$. 

For the Weilbull distribution, $F(x)=(1-\exp(-(x/m)^{\alpha}))\mathbf{1}(x\geq0)$, where $\alpha>0$ is referred to as the shape parameter and $m>0$ is the scale parameter. When the shape parameter is smaller than 1, it has DHR; when the shape parameter is equal to 1, it is an exponential distribution; when the shape parameter is larger than 1, it has IHR. We observe from the right plot in Figure \ref{SRPTWorse} that for small enough values of the shape parameter (i.e., $<0.6$), FCFS can achieve a shorter expected waiting time than SRPT. 

For the Pareto distribution, $F(x)=(x/m)^{\alpha}\mathbf{1}(x\geq m)$, where $\alpha>1$ is referred to as the shape parameter and $m>0$ is the scale parameter. The hazard rate function is no longer monotone on $(0,\infty)$. In particular, $h(x)=0$ for $x<m$, and $h(x)>0$ and is decreasing in $x$ for $x\geq m$. In this case, FCFS may not be optimal among blind policies (Proposition 5 in \cite{bassamboo2016scheduling}). We observe from the left plot in Figure \ref{SRPTWorse} that FCFS only achieves a shorter expected waiting time than LCFS for small enough values of the shape parameter (i.e., $<1.3$). Moreover, SRPT leads to shorter expected waiting times than both FCFS and LCFS in this case.


\begin{figure}[htp]
\centering
\includegraphics[width=.48\textwidth]{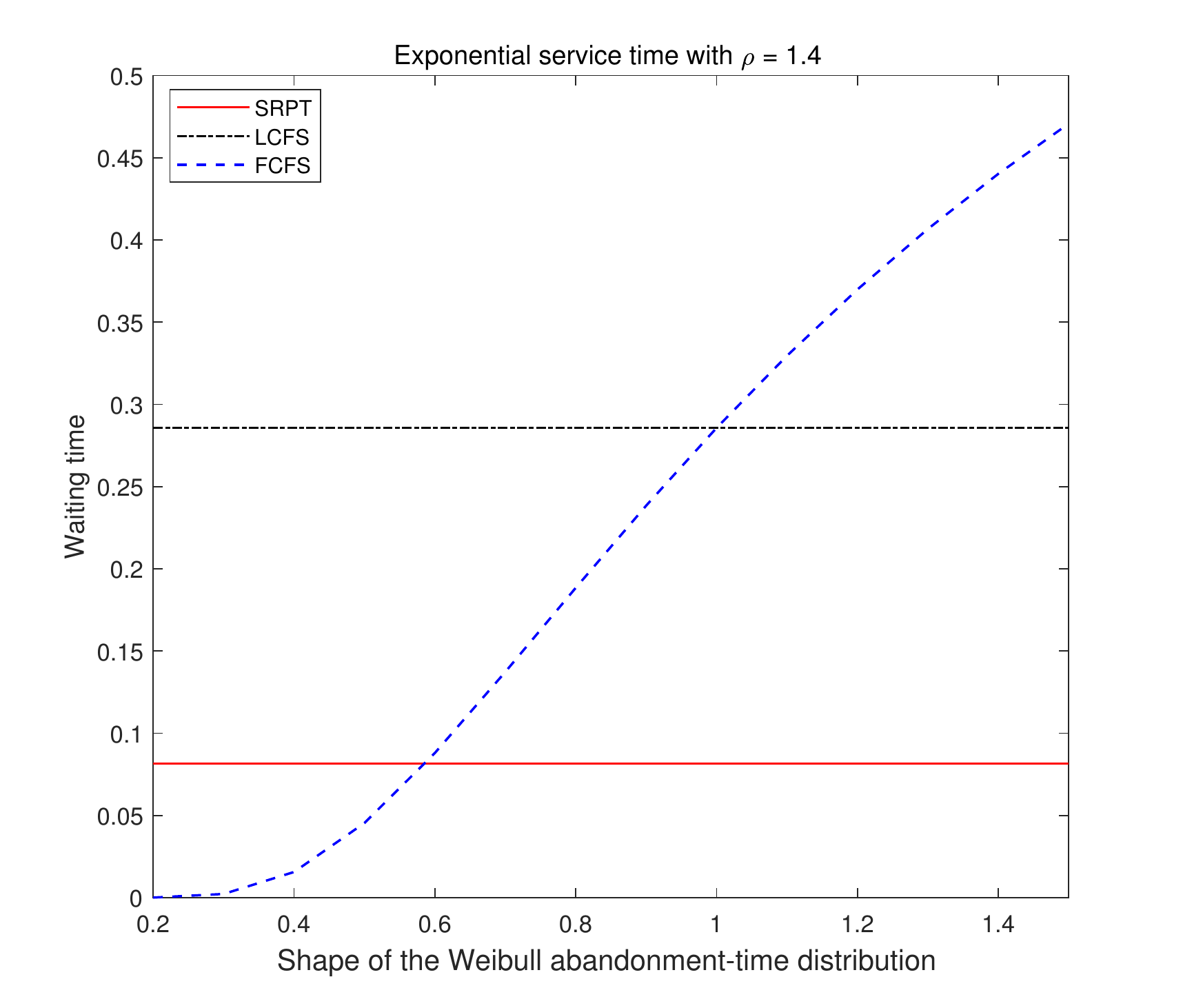}
\includegraphics[width=.48\textwidth]{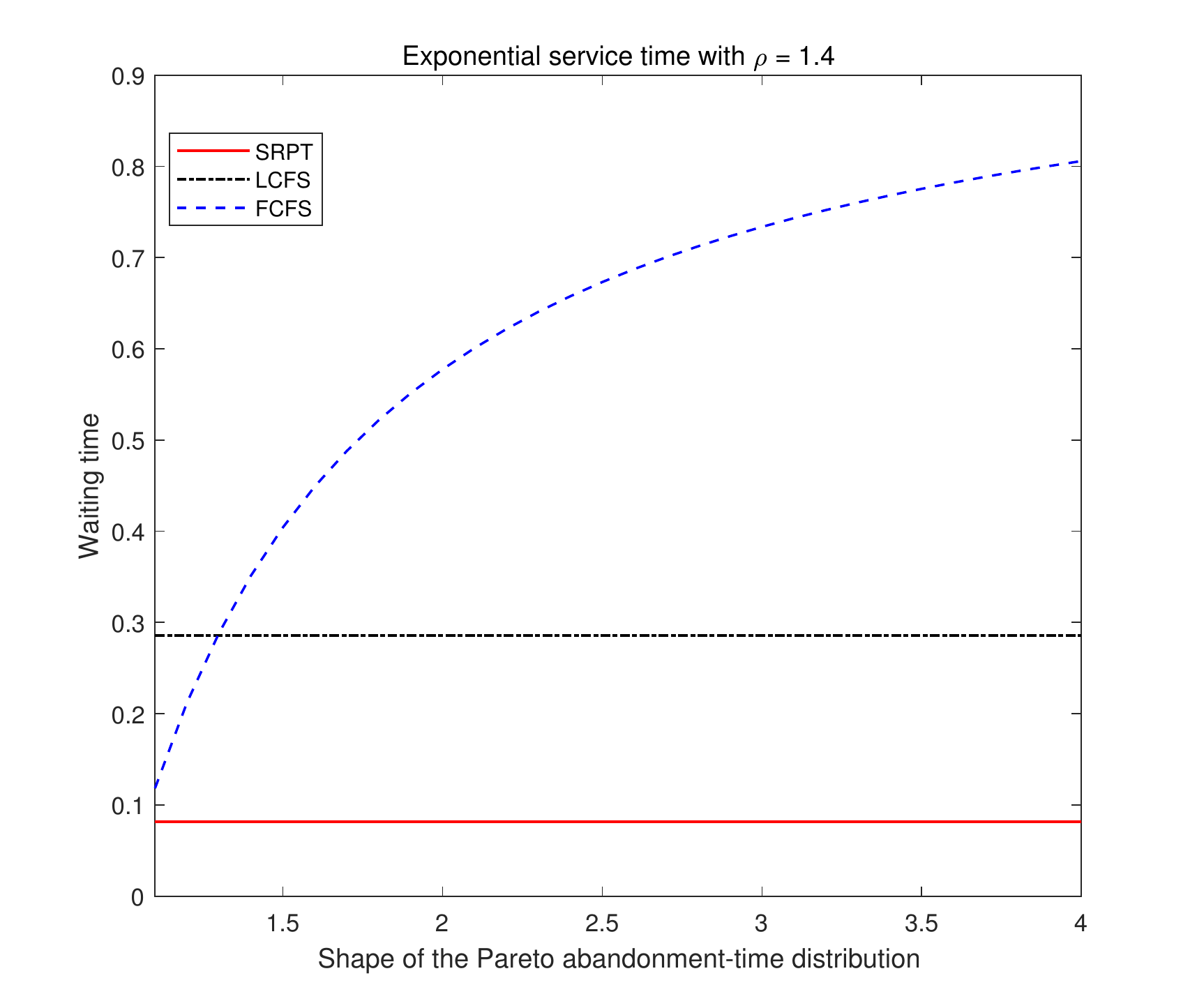}
\caption{\label{SRPTWorse} Steady-state waiting times for Weibull and Pareto abandonment under different shape parameters.  }
\end{figure}

\subsection{The Effect of the Service-Time Distribution} \label{STDist} It has been observed that in single-server queues, the advantage of SRPT is especially pronounced with a heavy tailed service-time distribution; see, for example, \cite{lin2011heavy} and \cite{chen2020power}. We next investigate whether the same holds in the multiserver setting with abandonment. 
For service times, we consider the family of Weibull distributions and the family of Pareto distributions, both with varying shape parameters. We set the patience-time distribution to be Weibull. Note that the service-time distribution has no effect on the throughput and steady-state fluid waiting times for FCFS and LCFS queues, but it plays an important role in the performance of the SRPT queue. 


In Figure \ref{throughputST} we compare the throughput for SRPT, FCFS, and LCFS when service times have Weibull (left-side figure) or Pareto (right-side figure) distributions. 
We vary the value of the shape parameter in the Weibull and Pareto service-time distributions and adjust the scale parameter accordingly so that the mean service time is fixed at 1. Note that as the shape parameter decreases, the tails of the Weibull or Pareto distributions become heavier (i.e., $1-F(x)$ decays to zero at a slower rate as $x$ increases). The patience-time distribution is fixed as a Weibull distribution with shape 0.4 and mean 1. We observe from the figure that SRPT always yields the highest throughput (as we proved in Proposition \ref{lm:th}). More importantly, the throughput of SRPT decreases as the shape parameter of the service-time distribution increases. This suggests that the heavier tail of the service time distribution leads to higher throughput in the SRPT queue. We also recall that the throughput is not affected by the patience-time distribution, so we do not reproduce identical figures for other patience-time distributions. 


In Figure \ref{waitsweibullST}, we compare the steady-state expected waiting times for SRPT, FCFS, and LCFS when service times have a Weibull distribution. Since the patience-time distribution affects the waiting times of the FCFS queue, we show two Weibull patience-time distributions with different shape parameters: 0.4 (left-side figure) and 1 (right-side figure). We first observe that when the patience-time distribution has DHR (i.e., Weibull with shape 0.4) and the service-time distribution has a large enough shape parameter (i.e., $>0.5$), FCFS can lead to a shorter steady-state average waiting time than SRPT. We also observe that similar to the throughput, for the steady-state waiting times, the greatest advantage of SRPT over FCFS or LCFS is achieved by service-time distributions with smaller shape parameters, i.e., corresponding to heavier tails. We make similar observations for the case of Pareto service-time distributions with varying shape parameters. To avoid repetition, we do not include the figures here.

\begin{figure}[htp]
\centering
\includegraphics[width=.48\textwidth]{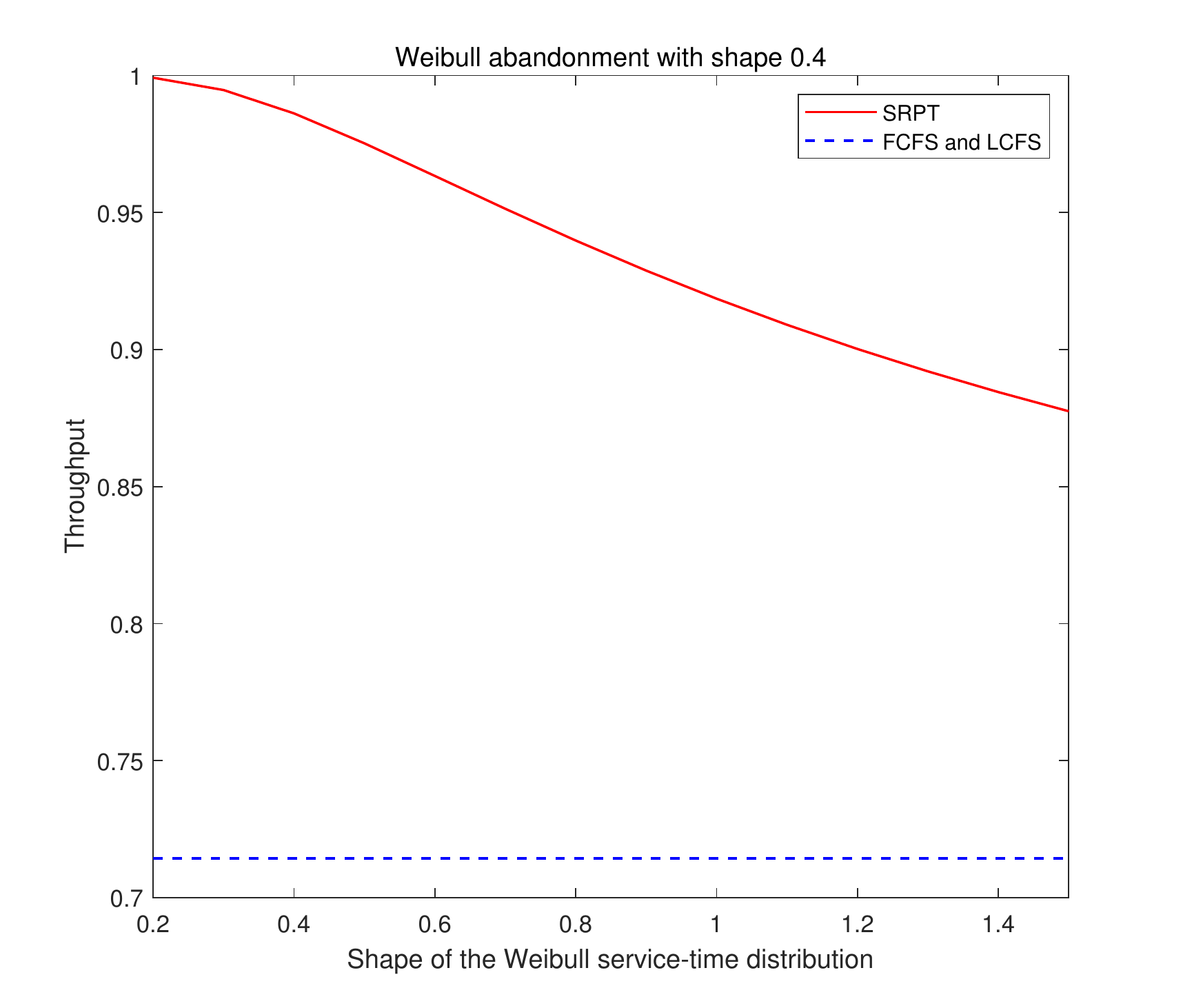}
\includegraphics[width=.48\textwidth]{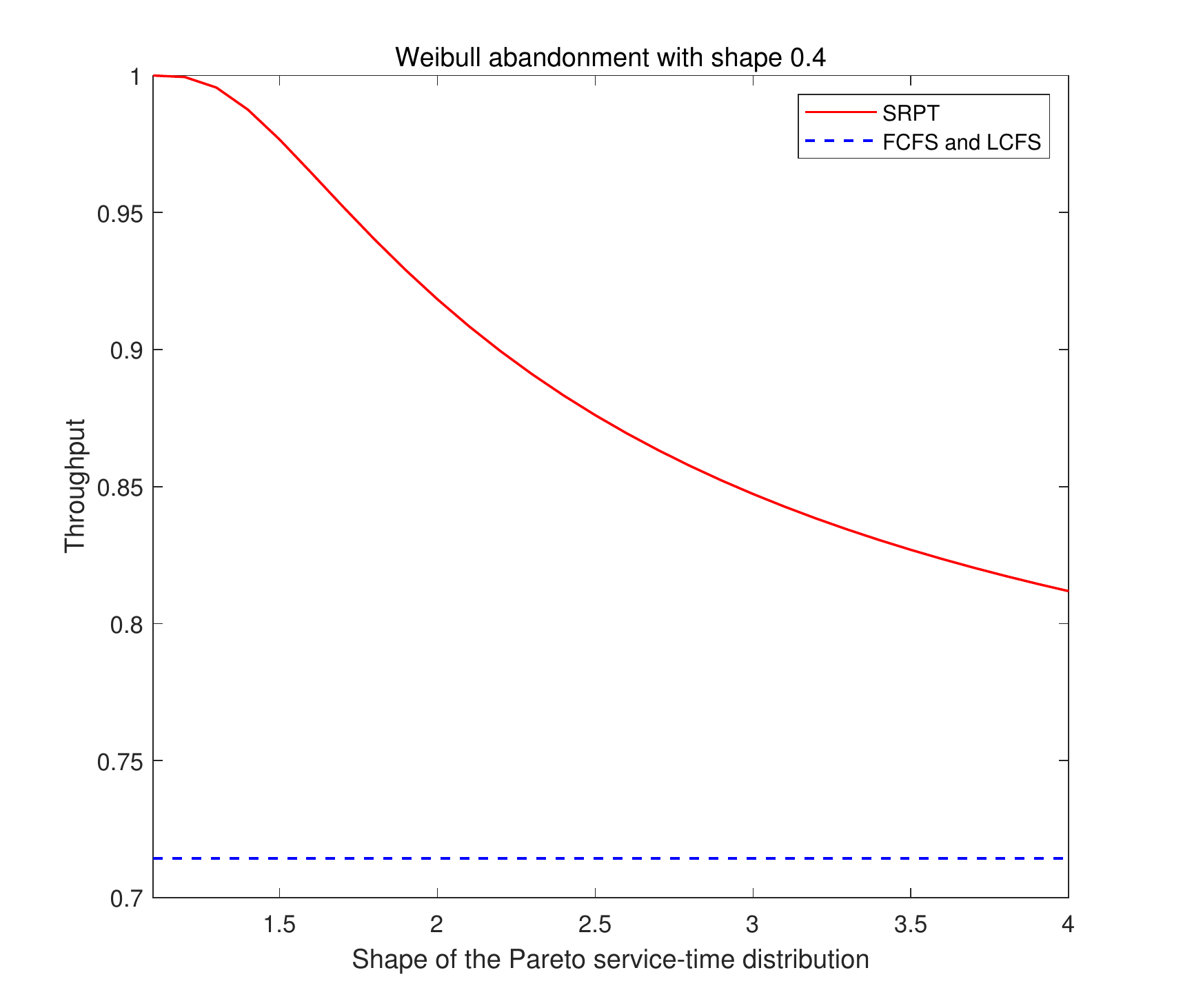}
\caption{\label{throughputST} Throughput for Weibull and Pareto service times under different shape parameters. }
\end{figure}

\begin{figure}[htp]
\centering
\includegraphics[width=.48\textwidth]{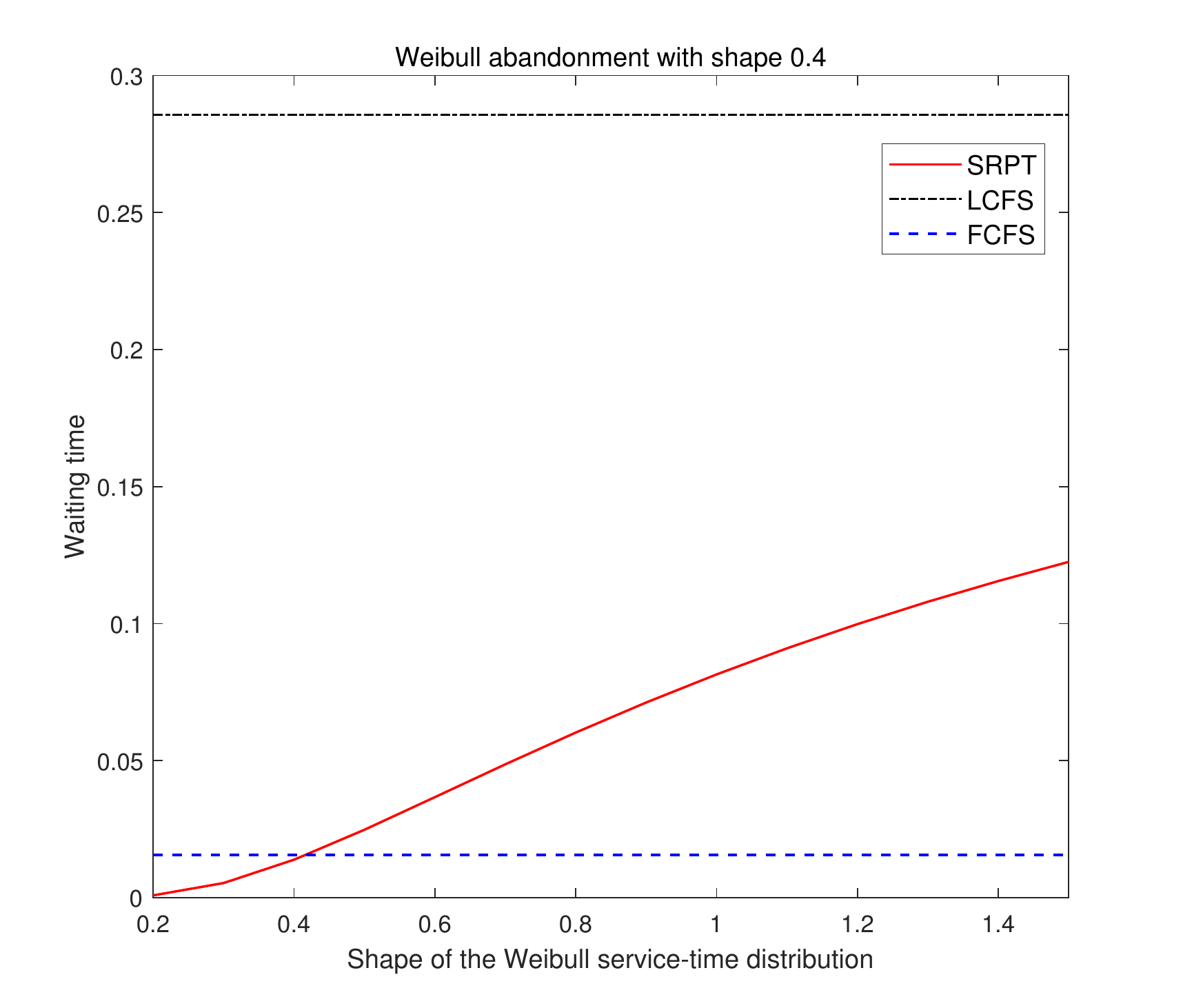} 
\includegraphics[width=.48\textwidth]{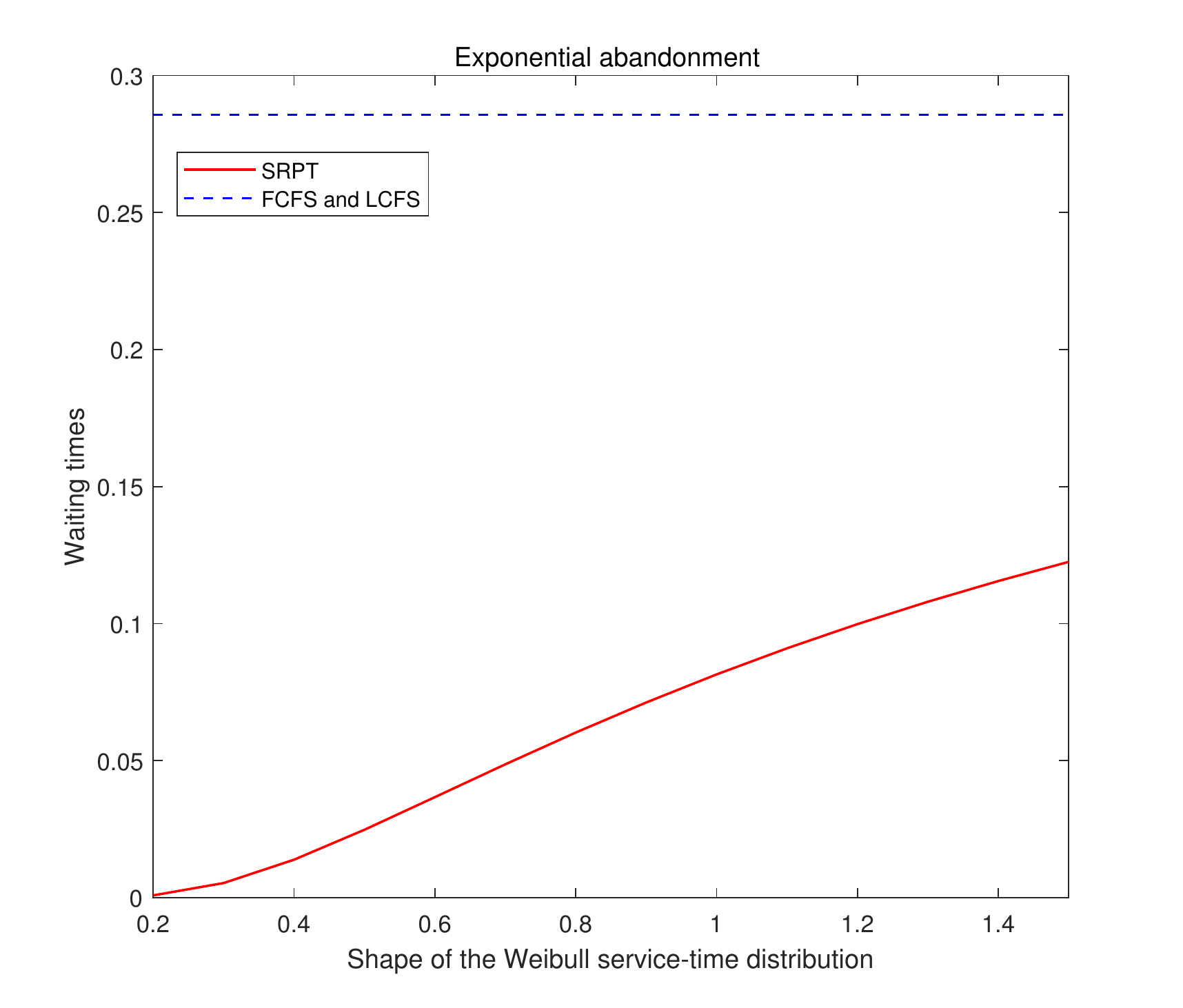}
\caption{\label{waitsweibullST} Steady-state waiting times for Weibull service times under different shape parameters. }
\end{figure}



\section{Conclusions} \label{conc}

In this paper, we presented the first theoretical analysis of the SRPT policy in multiserver queues with abandonment. We relied on a many-server asymptotic mode of analysis, and proved that the system is well approximated, in the limit, by a preemptive two-class priority system where customers with small service times (below a threshold) are served immediately, and customers with long service times (above a threshold) are not served and eventually abandon. We also showed that SRPT maximizes system throughput asymptotically. Lastly, we compared the performance of SRPT to blind policies, such as FCFS and LCFS, to gain insight into the effects of patient-time and service-time distributions.

%
%
%
 \begin{appendix}
 \section{Proof of Theorem \ref{thm:srpt}}\label{pf:thm1}
 We apply the same many-server overloaded scaling described in section \ref{scaling} to both the SRPT $M/GI/s^{\lambda}+GI$ queue and the two-class preemptive $M/GI/s^{\lambda}/s^{\lambda}$ loss queue. Let $\Th_O^{\lambda}$ be the throughput of the $\lambda$-th $M/GI/s^{\lambda}+GI$ queue, which we refer to as system $O^{\lambda}$.

 \subsection{Probability of Service}
 We begin by proving part $(a)$ in the theorem. Let $\gamma^{\lambda}(x)$ denote the steady-state probability of successfully finishing service for a customer with initial service time $x$ in system $O^{\lambda}$. We note that
\[\Th_O^{\lambda}=\lambda \int_{0}^{\infty}g(x)\gamma^{\lambda}(x)dx,\]
i.e., in stationarity, the rate at which customers finish service is the same as the rate at which they enter service.
In addition,
\begin{equation}\label{eq:little}
\lambda \left(\int_{0}^{\infty}g(x)\gamma^{\lambda}(x)dx\right)\frac{\int_{0}^{\infty}xg(x) \gamma^{\lambda}(x)dx}{\int_{0}^{\infty}g(x)\gamma^{\lambda}(x)dx} \leq s^{\lambda},
\end{equation}
i.e., the stationary number of customers in service is less than the service capacity.
Lastly, define: 
\begin{equation} \label{new} \bar \gamma(x)\equiv \lim_{\lambda\rightarrow\infty}\gamma^{\lambda}(x),~~\Ec_{\gamma}\equiv \int_{0}^{\infty}g(x)\bar \gamma(x)dx=\lim_{\lambda\rightarrow\infty}\frac{\Th_O^{\lambda}}{\lambda},~~\mbox{and}~~ \Sc_{\gamma}=\frac{\int_{0}^{\infty}xg(x)\bar \gamma(x)dx}{\int_{0}^{\infty}g(x) \bar \gamma(x)dx}.  \end{equation} 
From \eqref{eq:little}, and defining $\bar s \equiv (\rho\mu)^{-1}$, we have that:
\[ \Ec_{\gamma} \cdot \Sc_{\gamma} \leq \bar s.\]

We will need the following two lemmas, Lemma \ref{lm:th_loss} and Lemma \ref{lm:max}, which we state and prove. 

\begin{lemma}\label{lm:th_loss}
For the sequence of  two-class $M/GI/s^{\lambda}/s^{\lambda}$ queues under the preemptive priority rule,
\[\lim_{\lambda \rightarrow\infty}\frac{\Th_{L,1}^{\lambda}}{\lambda}= G(\tau).\]
\end{lemma}
\begin{proof}
We first show that for a fixed $\lambda$,
\[\Th_{L,1}^{\lambda}=\lambda G(\tau)\left(1-\left(\int_{0}^{\infty}(1+t/s^\lambda)^{s^\lambda} e^{-t}dt\right)^{-1}\right).\]
As class 1 customers have preemptive priority over class 2 customers: For a class 1 customer, the system operates like a single-class loss queue with arrival rate $\lambda G(\tau)$ and service-time distribution $[S|S\leq \tau]$. As the steady-state blocking probability depends on the service-time distribution only through its mean, and $r=\lambda G(\tau)\E[S|S\leq \tau]=s^{\lambda}$, we have the steady-state blocking probability:
\[\begin{split}
\Pb^{\lambda}:=\frac{r^{s^{\lambda}}/s^{\lambda}!}{\sum_{k=0}^{s^{\lambda}}r^k/k!}&=\left(\sum_{k=0}^{s^{\lambda}}\frac{s^{\lambda}!}{k!}r^{-(s^{\lambda}-k)}\right)^{-1}\\
&=\left(\sum_{k=0}^{s^{\lambda}}\frac{s^{\lambda}!}{(s^{\lambda}-k)! k!}k!r^{-k}\right)^{-1}\\
&=\left(\sum_{k=0}^{s^{\lambda}}{s^{\lambda} \choose k} r^{-k}\int_{0}^{\infty}t^k e^{-t}dt\right)^{-1}\\
&=\left(\int_{0}^{\infty}\sum_{k=0}^{s^{\lambda}}{s^{\lambda} \choose k} r^{-k}t^k e^{-t}dt\right)^{-1}\\
&=\left(\int_{0}^{\infty}(1+t/r)^{s^{\lambda}} e^{-t}dt\right)^{-1}=\left(\int_{0}^{\infty}(1+t/s^{\lambda})^{s^{\lambda}} e^{-t}dt\right)^{-1}.
\end{split}\]
Next, class 1 customers enter service in the loss queue at rate $\lambda G(\tau)(1-\Pb^{\lambda})$. By rate conservation, we have
$\Th_{L,1}^{\lambda}=\lambda G(\tau)(1-\Pb^{\lambda})$.
Then,
\[
\lim_{\lambda\rightarrow\infty}\frac{\Th_{L,1}^{\lambda}}{\lambda}
=G(\tau)\left(1-\lim_{\lambda\rightarrow\infty} \left(\int_{0}^{\infty}(1+t/s^{\lambda})^{s^{\lambda}} e^{-t}dt\right)^{-1}\right)=G(\tau).
\]
\end{proof}

\begin{lemma} \label{lm:max}
For the sequence of $M/GI/s^{\lambda}+ GI$ queues under SRPT, $\bar \gamma(x)=1\{x\leq \tau\}$.
\end{lemma}
\begin{proof}
We first note that Lemma \ref{lm:th_loss} implies that
\[
\Ec_{\gamma}=\lim_{\lambda\rightarrow\infty} \frac{\Th_{O}^{\lambda}(\infty)}{\lambda}
\geq \lim_{\lambda\rightarrow\infty}\frac{\Th_{L,1}^{\lambda}(\infty)}{\lambda} = G(\tau)
\]
As $\Ec_{\gamma} \cdot \Sc_{\gamma} \leq \bar s$,
\[\Sc_{\gamma} \leq \E[S|S\leq \tau].\]

Second, consider the optimization problem
\begin{equation}\label{eq:max}
\begin{split}
\min_{\gamma}& \int_{0}^{\infty} xg(x)\gamma(x)dx\\
\mbox{s.t.} ~ & \int_{0}^{\infty}g(x)\gamma(x)dx \geq \int_{0}^{\tau}g(x)dx\\
&\gamma(x)\in [0,1] \mbox{ for all $x\in [0,\infty)$}
\end{split}
\end{equation}
The constraint $\int_{0}^{\infty}g(x)\gamma(x)dx \geq \int_{0}^{\tau}g(x)dx$ implies that
\[\int_{\tau}^{\infty}g(x)\gamma(x)dx\geq \int_{0}^{\tau}g(x)(1-\gamma(x))dx.\]
For the objective function, we note that
\[\begin{split}
\int_{0}^{\infty} xg(x)\gamma(x)dx &\geq \int_{0}^{\tau} xg(x)\gamma(x)dx + \tau\int_{\tau}^{\infty} g(x)\gamma(x)dx\\
&\geq \int_{0}^{\tau} xg(x)\gamma(x)dx + \tau \int_{0}^{\tau}g(x)(1-\gamma(x))dx \geq \int_{0}^{\tau} xg(x)dx.
\end{split}\]
When $\gamma(x)=1\{x\leq \tau\}$, $\int_{0}^{\infty} xg(x)\gamma(x)dx=\int_{0}^{\tau} xg(x)dx$.
Thus, the minimum of \eqref{eq:max} is equal to $\E[S1\{S\leq \tau\}]$ and the minimum is achieved when
$\gamma(x)=1\{x\leq T\}$.

Lastly, as $\int_{0}^{\infty}g(x)\bar \gamma(x)dx \geq G(\tau)$ and
$\int_{0}^{\infty}xg(x)\bar \gamma(x)dx=\E[S|S\leq\tau]$, $\bar \gamma(x)=1\{x\leq \tau\}$.
\end{proof}

Lemma \ref{lm:max} implies that
\[\lim_{\lambda\rightarrow\infty} \frac{\Th_{O,1}^{\lambda}(\infty)}{\lambda}=G(\tau) \mbox{ and }
\lim_{\lambda\rightarrow\infty} \frac{\Th_{O,2}^{\lambda}(\infty)}{\lambda}=0,\]
which further implies that
\[\lim_{\lambda \rightarrow \infty} \PP(\Serv_c^{\lambda}| S_c \leq \tau) = 1 \mbox{ and }
\lim_{\lambda \rightarrow \infty} \PP(\Serv_c^{\lambda}| S_c > \tau) = 0.\]
We have thus proved part $(a)$ in Theorem \ref{thm:srpt}.

\subsection{Other Performance Measures}
We now turn to proving parts $(b)-(d)$ in the theorem. 

\begin{proof}

For customers with service time less than or equal to $\tau$, by part $(a)$ in Theorem \ref{thm:srpt},
\[\PP(T_c>V_c^{\lambda}|S_c\leq \tau) = \PP(\Serv_c^{\lambda}|S_c\leq \tau) \rightarrow 1 \mbox{ as $\lambda\rightarrow \infty$}. \]
We note from the proof of part (a) that the convergence holds regardless of patience time distribution. Consider a patience time distribution with $f(0)>0$, e.g., exponential patience time.
Then, we have $(V_c^{\lambda}|S_c\leq \tau)\Rightarrow 0$ as $\lambda\rightarrow \infty$.
Because $0\leq W_c^{\lambda}\leq V_c^{\lambda}$, $(W_c^{\lambda}|S_c\leq \tau)\Rightarrow 0$.
Next, as $W_c^{\lambda}=\min\{V_c^{\lambda}, T_c\}$, $W_c^{\lambda}\leq T_c$, and as $\E[T_c]=1/\theta<\infty$, by dominated convergence theorem,
\[\E[W_c^{\lambda}|S_c\leq \tau]\rightarrow 0 \mbox{ as $\lambda\rightarrow \infty$.}\]

For customers with service time larger than $\tau$, by part $(a)$ in Theorem \ref{thm:srpt},
\[\PP(T_c \leq V_c^{\lambda}|S_c\leq \tau) = \PP(\Ab_c^{\lambda}|S_c\leq \tau) \rightarrow 1 \mbox{ as $\lambda\rightarrow \infty$}. \]
This implies that $(W_c^{\lambda}|S_c> \tau)\Rightarrow T_c$. 
Because $W_c^{\lambda}\leq T_c$ and $\E[T_c]<\infty$, by dominated convergence theorem,
\[\E[W_c^{\lambda}|S_c> \tau]\rightarrow \E[T_c]=1/\theta \mbox{ as $\lambda\rightarrow \infty$.}\]

Lastly, the convergence of $\E[W_c^{\lambda}]$ follows from the fact that $\E[W_c^{\lambda}]=\E[W_c^{\lambda}|S_c\leq \tau]\PP(S_c\leq \tau)+\E[W_c^{\lambda}|S_c> \tau]\PP(S_c>\tau)$,

\end{proof}
\section{Proof of Proposition \ref{lm:th}} \label{app:prop2}

\begin{proof}
Let $\gamma(x)$ denote the steady-state probability of getting served for a customer with service time requirement $x$.
We first note that for the $M/GI/s^\lambda+GI$ queue, any scheduling policy in stationarity must satisfy
\[
\left(\lambda \int_{0}^{\infty}g(x)\gamma(x)dx\right)\cdot \frac{\int_{0}^{\infty}xg(x)\gamma(x)dx}{\int_{0}^{\infty}g(x)\gamma(x)dx}
\leq s^{\lambda} = \lambda \int_{0}^{\tau}xg(x)dx
\]
Thus, the maximum throughput of the system can be upper bounded by the optimal value
of the following optimization problem.
\begin{equation}\label{eq:max2}
\begin{split}
\max_{\gamma}~ & \lambda \int_{0}^{\infty} g(x)\gamma(x)dx\\
\mbox{s.t.} ~ & \int_{0}^{\infty}xg(x)\gamma(x)dx \leq  \int_{0}^{\tau}xg(x)dx\\
&\gamma(x)\in [0,1] \mbox{ for all $x\in [0,\infty)$}
\end{split}
\end{equation}
The constraint $\int_{0}^{\infty}xg(x)\gamma(x)dx \leq \int_{0}^{\tau}xg(x)dx$ implies that
\[
\int_{0}^{\tau}xg(x)\gamma(x)dx + \tau\int_{\tau}^{\infty} g(x)\gamma(x)dx
\leq \int_{0}^{\infty}xg(x)\gamma(x)dx \leq \int_{0}^{\tau}xg(x)dx,
\]
which further implies that
\[\begin{split}
&\tau\int_{\tau}^{\infty} g(x)\gamma(x)dx \leq \int_{0}^{\tau}xg(x)(1-\gamma(x))dx
\leq \tau \int_{0}^{\tau}g(x)(1-\gamma(x))dx,\\
& \int_{\tau}^{\infty} g(x)\gamma(x)dx \leq \int_{0}^{\tau}g(x)(1-\gamma(x))dx.
\end{split}\]
For the objective function, we have
\[\begin{split}
\lambda \int_{0}^{\infty} g(x)\gamma(x)dx&=\lambda \left(\int_{0}^{\tau} g(x)\gamma(x)dx+\int_{\tau}^{\infty} g(x)\gamma(x)dx\right)\\
&\leq \lambda \left(\int_{0}^{\tau} g(x)\gamma(x)dx+\int_{0}^{\tau}g(x)(1-\gamma(x))dx\right)\\
&=\lambda \int_{0}^{\tau} g(x) dx=\lambda G(\tau).
\end{split}\]

In addition, for $\gamma^*(x)=1\{x\leq \tau\}$,
\[\lambda \int_{0}^{\infty}g(x)\gamma^*(x)dx =\lambda G(\tau) \mbox{ and }
\int_{0}^{\tau}xg(x)\gamma^*(x)dx =\int_{0}^{\tau}xg(x)dx.\]
This implies that $\gamma^*$ is an optimal solution to \eqref{eq:max2}
with the optimal objective value $\lambda G(\tau)$.
Thus, $\Th_M^{\lambda} \leq \lambda G(\tau)$.
Because $\lim_{\lambda\rightarrow\infty} \Th_O^{\lambda}/\lambda = G(\tau)$,
SRPT maximizes the throughput asymptotically.
\end{proof}

 \end{appendix}

\bibliographystyle{plain}
\bibliography{./references}

\end{document}